\documentclass[11pt]{amsart}

\usepackage[margin=1.2in]{geometry}

 \usepackage{comment}

 \usepackage{graphicx}
\usepackage{bbm}
\usepackage[title]{appendix}
\usepackage{color}
\usepackage{tikz-cd}

 \usepackage{graphicx}
\usepackage{tikz-cd}

\newcommand{\bea}{\begin{eqnarray}}
\newcommand{\eea}{\end{eqnarray}}

\allowdisplaybreaks

\newcommand{\Z}{\Bbb Z}

\newcommand{\C}{\Bbb C}

\newcommand{\g}{  \mathfrak g}

\newtheorem{theorem}{Theorem}[section]
\newtheorem{lemma}[theorem]{Lemma}

\theoremstyle{definition}
\newtheorem{definition}[theorem]{Definition}

\theoremstyle{remark}
\newtheorem{remark}[theorem]{Remark}

\theoremstyle{proposition}
\newtheorem{proposition}[theorem]{Proposition}
\theoremstyle{corollary}
\newtheorem{corollary}[theorem]{Corollary}
\numberwithin{equation}{section}

\def \<{\langle}
\def \>{\rangle}

\begin{document}

\title[]{  Nappi-Witten vertex operator algebra    via inverse Quantum Hamiltonian Reduction
  }
\author[]{Dra\v zen  Adamovi\' c}
\author[]{Andrei Babichenko}

\maketitle

\begin{abstract}
The study of representation theory of the Nappi-Witten VOA   was initiated in   \cite{BJP} and \cite{BKRS}.   In this paper we use the 
 technique of inverse quantum hamiltonian reduction to investigate the representation theory of the Nappi-Witten VOA $ V^1(\mathfrak h_4)$.
We first prove    that the quantum hamiltonian reduction of  $ V^1(\mathfrak h_4)$ is   the Heisenberg-Virasoro VOA $L^{HVir}$ of level zero investigated in  \cite{AR, Billig}.
We   invert  the quantum hamiltonian reduction in this case and  prove that   $ V^1(\mathfrak h_4)$ is  realized as a vertex subalgebra of $L^{HVir} \otimes \Pi$, where   $\Pi$ is a certain lattice-like vertex algebra. Using such an  approach we shall realize all relaxed highest weight modules which were classified in \cite{BKRS}. We show that  every relaxed highest weight module, whose top components is neither highest nor lowest weight $\mathfrak h_4$--module, has the form $M_1 \otimes \Pi_{1} (\lambda)$ where $M_1$ is an irreducible, highest weight $L^{HVir}$--module and $\Pi_{1} (\lambda)$ is an irreducible weight $\Pi$--module.
 Using the fusion rules for $L^{HVir}$--modules from \cite{AR}  and the previously developed   methods of constructing logarithmic modules from  \cite{AdM-selecta} we are able to construct a family of logarithmic $V^1(\mathfrak h_4)$-modules.    The Loewy diagrams  of these logarithmic modules are completely analogous to the Loewy diagrams of  projective modules  of weight $L_k(\mathfrak{sl}(2))$--modules, so  we  expect that our logarithmic modules  are also  projective in  a certain category of weight  $ V^1(\mathfrak h_4)$--modules.
\end{abstract}

\section{Introduction}

\subsection{Physics background}
In recent years, there has been an increased interest in studying two dimensional conformal theories
associated to non-semisimple Lie algebras. Among them are field theories based on abelian extensions
of the Virasoro algebras, such as the Heisenberg-Virasoro algebra, Gallilean conformal algebra (relevant for
the non-relativistic $AdS_2/CFT_1$ correspondence) and the
Virasoro algebra extended by spin two quasi-primary field with regular OPE with itself (also known
as the $BMS_3$ algebra). The latter appears as the symmetry of the tensionless closed bosonic string worldsheet.

Another important set of examples comes from two dimensional conformal field theories on manifolds
with non-reductive Lie groups. One of the simplest among them is
the Nappi-Witten WZW model \cite{NW}, which was originally introduced in order to describe string
propagating in monochromatic plane-wave background (particularly
in the presence of gravitational waves \cite{KK,DAK}). Its abelian and diagonal cosets have been related to other
interesting backgrounds \cite{DAQ1,DAQ2}. An understanding of the representation theory of two dimensional
conformal symmetry of the Nappi-Witten model, beyond the affinization of unitary representations
of its Lie algebra, is therefore an important problem which can be investigated by means of
vertex operator algebra (VOA) methods. This is the subject of the present paper.

\subsection{ The Nappi-Witten vertex algebra: state of the art}

Nappi-Witten Lie algebra  $\mathfrak h_4$ is the four-dimensional complex Lie algebra with basis $\{E,F, I,J\}$ whose nonzero Lie
brackets are, modulo antisymmetry, as follows:
$$  [E,F] = I, \  [J,E] = E,\  [J,F] = -F. $$
Note that $I$ is in the center of $\mathfrak h_4$, and  $\mathfrak h_4$  is not reductive Lie algebra. Let $$\widehat{ \mathfrak h}_4= \mathfrak h_4 \otimes {\C}[t, t^{-1}] + {\C} K$$ be the associated affine Lie algebra.  Let $V^k(\mathfrak h_4)$ be the universal affine vertex algebra. The construction is the same as that of $V^k(\g)$, where $\g$ is a simple Lie algebra (cf. \cite{Kac, LL}). One can show (cf. \cite{BKRS}) that for $k \ne 0$,  $V^k(\mathfrak h_4)$ is a simple vertex algebra and that $V^k(\mathfrak h_4) \cong V^1(\mathfrak h_4)$.

The representation theory of affine Nappi-Witten algebra was previously addressed
using vertex operator algebra theory  methods  in \cite{BJP} and \cite{BKRS}. In  \cite{BJP} a necessary
and sufficient condition for generalized Verma module to be irreducible was formulated
and irreducible highest weight modules were classified. Furthermore, a Wakimoto type realization of
Nappi-Witten VOA was suggested, which we use in the present paper.
A more general approach to this representation theory was worked out in  \cite{BKRS}, with the emphasis
on the special role of relaxed highest weight modules and their classification \cite{KR}. Specifically, the authors studied the category of weight modules with finite dimensional weight spaces over the associated affine VOA, obtaining a classification of the irreducible modules in this category  and computing their characters.
Non-semisimplicity of this category was shown, meaning that it contains indecomposable modules, however their explicit construction was not implemented.

\subsection{ Quantum Hamiltonian reduction and Heisenberg-Virasoro vertex algebra $L^{HVir}$}
 For any simple Lie algebra $\g$   and pair $(x,f)$ such that $ad(x)$  is diagonalizable on $\g$ and $[x,f] = -f$ and $f$ is an even nilpotent element, Kac,  Wakimoto and Roan  in \cite{ KW04, KWR} constructed the  universal affine $W$--algebra $W^k(\g, x, f)$.  In the present paper, we show that this construction can be applied for the Nappi-Witten algebra $\g = \mathfrak h_4$. We take $x= J$, and $f = F$. Then   $ad(x)$ defines the gradation  $$\g = \g_{-1} + {\g}_0 + \g _1, $$
 such that $\g ^f = \mbox{span}_{\C} \{F, I\}$, and using \cite{KW04}  we get our first new result (cf. Proposition  \ref{prop-ds}):

\begin{itemize}
\item $W^{k=1}(\g, x, f)$  is generated  by the Virasoro field $\hat L$  of central charge $c_L=2$ and commutative Heisenberg field $\hat I$ such that $\hat L$  and $\hat I$ generate the Heisenberg-Virasoro vertex algebra $L^{HVir}$ (cf.  Section \ref{Hvir-sect} for definition).
\end{itemize}
The Heisenberg Virasoro vertex algebra $L^{HVir}$ of level zero was introduced by Y. Billig \cite{Billig}, and its  vertex algebraic  properties were further investigated in \cite{AR}. Therefore the representation theory of $ V^1(\mathfrak h_4)$ should be closely related to the representation theory $L^{HVir}$. In this article, we explore this correspondence  in more details.

\subsection{    Inverse Quantum  Hamiltonian Reduction         }
Construction of inverses of QHR and their applications in the representation theory of affine vertex algebras and affine $W$-algebras is a relatively new topic in the vertex algebra theory (cf. \cite{A-2017, ACG-2024, AKR-2021, AKR-2023, Zac1, Zac2} ).
Most closely related to the Nappi-Witten vertex algebra is the case  $\g = \mathfrak{sl}_2$, for which the construction and its applications  were studied in \cite{A-2017}. Then the (simple) affine vertex algebra $L_k(\mathfrak{sl}_2)$ is realized as a subalgebra of  $L^{Vir} _{c_k} \otimes \Pi$, where $L^{Vir} _{c_k}$ is the (simple) Virasoro vertex algebra  of central charge $c_k = 1- 12 \frac{(k+1) ^2}{k+2}$ and $\Pi$ is the half lattice vertex algebra (cf. Subsection \ref{pi-voa}). In the case of the  Nappi-Witten vertex algebra, we prove the following: (cf. Theorem  \ref{inverse-realization}):

\begin{itemize}
\item There exists an  embedding $ V^1(\mathfrak h_4) \hookrightarrow L^{HVir} \otimes \Pi$.
\item 

There is a homomorphism of $V^1(\mathfrak h_4)$--modules   $S :   L^{HVir} \otimes \Pi \rightarrow  M_2$, where $M_2$ is certain weight   $ L^{HVir} \otimes \Pi $--module,  such that  $ V^1(\mathfrak h_4)  =  \mbox{Ker} _{ L^{HVir} \otimes \Pi}  (S)$. The operator $S$ is also called the screening operator.
\end{itemize}

\subsection{Relaxed and logarithmic modules}

In \cite{A-2017} all  irreducible, relaxed highest weight modules { over admissible affine vertex algebras $L_k(\mathfrak{sl}_2)$}  have been shown to have the form $M_1 \otimes M_2$, where $M_1$ is an irreducible $L^{Vir} _{c_k} $--module, and  $M_2$ is an irreducible $\Pi$--module. The paper \cite{A-2017} also contains a realization of logarithmic modules, which were later identified in \cite{ACK-23} as projective covers of irreducible modules in a certain category.  Their structure is equivalent to the  structure of projective modules for the corresponding (unrolled) small quantum group. We believe that similar results can be obtained for  $ V^1(\mathfrak h_4)$--modules. In this paper we make the first step in this direction and present a realization of a family of logarithmic $ V^1(\mathfrak h_4)$--modules following the approach from \cite{A-2017}.

Let us explain our result in more details.
\begin{itemize}
\item Let  $L^{HVir}[x,y]$ be  a highest weight $L^{HVir}$--module  (cf.  Section \ref{Hvir-sect}) and $\Pi_r(\lambda)$ a weight $\Pi$--module (cf.  Subsection \ref{pi-voa}).

\item  We prove  in Theorem \ref{thm:relax} that all irreducible relaxed modules for $ V^1(\mathfrak h_4)$--modules, whose top component is neither lowest nor highest weight $\mathfrak h_4$--module, have the form
$L^{HVir}[x,y] \otimes \Pi_1(\lambda)$.
\item  Then we  consider  the extended vertex algebra $\mathcal V=  L^{HVir} \otimes \Pi \oplus L^{HVir}[-1,0] \otimes \Pi_{-1}(0)$, and using   Proposition  \ref{log-koncept}  we see that any  $\mathcal V$--module $(M, Y_M)$ can be deformed to   a structure of logarithmic  $V^1(\mathfrak h_4)$--module $(\widetilde M, Y_{\widetilde M})$ such that the deformed action of the Virasoro algebra is
$\widetilde L(z) = L^{sug} (z) +  z^{-1} Y(s, z) $, 
{ where  $s$ is a primary vector for the action of the Virasoro algebra of conformal weight $1$ of the form $s = v_{-1,0} \otimes e^{\nu}  $ 
such that $v_{-1,0}$ is the highest weight vector of $L^{HVir}[-1,0]$ and $e^{\nu} \in   \Pi_{-1}(0)$. }

\item Using the fusion rules for the Heisenberg-Virasoro vertex algebra $L^{HVir}$ and the fusion rules for  $\Pi$--modules, we construct  a family of $\mathcal V$--modules $U[x,y, \lambda, r]$, which after deformation  produces the  logarithmic modules
$\mathcal P_r[x,y]$.
\item We prove in  Theorem  \ref{log-non-split}  that   $\mathcal P_r[x,y]$ is a non-split extension
$$ 0   \rightarrow L^{HVir} [x-1,  \frac{x-2}{x-1} y] \otimes \Pi_{r-1}(\lambda)    \rightarrow   \mathcal P_{r}[x,y] \rightarrow
 L^{HVir} [x, y] \otimes \Pi_r(\lambda)
  \rightarrow 0 $$
  of two reducible  weight  modules. In Subsection  \ref{str-log-2} we obtain the Loewy diagram of $  \mathcal P_r[x,y]$.
  \item We also present some conjectures about  the projectivity of these modules.

\end{itemize}

\subsection*{Acknowledgements} We thank  referee for helpful suggestions and comments.

	D. Adamovi\' c  was supported by the Croatian Science Foundation     under the project IP-2022-10-9006.
 A. Babichenko was partially supported by the ISF Grant 1957/21 and thankful to Department of Mathematics, University of Zagreb for hospitality where this project was started.

\section{Preliminaries}

\subsection{Vertex algebras and their  logarithmic modules}

\label{constr-log-uvod}

 In the paper we assume that the reader is familiar with basic concepts in the vertex algebra theory such as modules, intertwining operators (cf. \cite{Kac, FZ,  LL}). We shall recall here only the definition of logarithmic modules.

Let $(V,Y,{\bf 1}, \omega)$ be a vertex operator algebra.  For $v \in V$, we set $Y(v,z) = \sum_{n \in {\Z}} v_n z^{-n-1}$. Let  $(M, Y_M)$ be its weak module.  
Then  the components of the field
$$ Y_M(\omega, z) = \sum_{n \in {\Z}} L(n) z^{-n-2} \quad (L(n) = \omega_{n+1}), $$
defines on $M$ the structure of a module for the Virasoro algebra.
\begin{definition}
A weak module     $(M, Y_M)$  for the vertex operator algebra $(V, Y , 1, {\omega})$  is called  a logarithmic module if it admits the following decomposition
   $$ M = \coprod_{r \in {\C}}  M_r, \quad  M_r = \{ v \in M \ \vert \ (L(0) -r) ^k = 0 \ \mbox{for some} \ k \in {\Z}_{>0} \}. $$

  If $M$ is a logarithmic module, we say that it has    rank 
   $m \in {\Z_{\ge 1} }$ if $$ (L(0)-L_{ss} (0) ) ^{m}  = 0, \quad (L(0)-L_{ss} (0) ) ^{m-1}   \ne 0,$$
where $L_{ss}(0)$ is the semisimple part of $L(0)$.
\end{definition}

Let $(V, Y_V , {\bf 1}, \omega)$ be a vertex operator algebra and $(M, Y_M)$ a $V$--module having integral weights with respect to $L(0) $.   Let $\mathcal V = V\oplus M$.
Define $$ Y_{\mathcal V} ( v_1 + w_1, z) ( v_2 + w_2) = Y_V  (v_1, z) v_2 +  Y _M (v_1, z ) w_2  + e^{ z L(-1) } Y_M( v_2, -z ) w_1, $$
where $v_1, v_2  \in V$, $w_1, w_2 \in M$.
Then by \cite{Li} $(\mathcal V,  Y_{\mathcal V} ,  {\bf 1}, \omega )$ is a vertex operator algebra.
The following lemma gives a method for a construction of a family of  $\mathcal V$--modules.

\begin{lemma}\cite{AdM-2012} \label{AdM-2012}
Assume that $(M_2, Y_{M_2}) $ and $(M_3, Y_{M_3})$ be $V$--modules, and let $\mathcal Y (\cdot, z)$ be an intertwining operator of type   $\binom{ M_3}{ M  \ \ M_2}$ with integral powers of  $z$. Then $(M_2 \oplus M_3, Y_{M_2 \oplus M_3})$ is a $\mathcal V$--module, where the vertex operator is  given by
$$ Y_{M_2 \oplus M_3} (v + w, z) ( w_2 + w_3) = Y_{M_2} (v, z) w_2 + Y_{M_3} (v, z) w_3   + \mathcal Y (w, z) w_2 $$
 for $v \in V$, $w \in M$, $w_i \in M_i$, $i=1,2$.
\end{lemma}

Now let  $s \in \mathcal V$ such that $L(n) s = \delta_{n,0} s$ for $n \in {\Z}_{\ge 0}$ and

\bea  [Y(s, z_1) , Y(s, z_2)] = 0. \label{uvjet-1}  \eea
Note that if $s \in M$, the condition (\ref{uvjet-1}) is always satisfied.

Then $S = s_0 = \mbox{Res}_z Y(s,z)$ is a screening operator, and $\overline V = \mbox{Ker}_V (S)$ is a vertex subalgebra of $V$. We recall the construction of logarithmic modules from \cite{AdM-selecta}.

\begin{proposition} \label{log-koncept} \cite{AdM-selecta} Assume that $(U, Y_U)$ is any $\mathcal V$--module. Then
$$ (\widetilde U, Y_{\widetilde U}(\cdot, z)) := (U, Y_U(\Delta(s,z) \cdot, z))$$
is a  $\overline V$--module, where
$$ \Delta (s, z) = z^{s_0} \exp \left( \sum_{n=1} ^{\infty} \frac{s_n}{-n} (-z) ^{-n}\right). $$
In particular,
$\widetilde L(z) = \widetilde Y_{\widetilde M} (\omega  , z) = L  (z) + z^{-1} Y_M(s, z)$,
and $\widetilde L(0) = L(0) + S$.
\end{proposition}

\subsection{Weyl vertex algebra $W$}
\label{sub-sect-weyl} Recall that the   Weyl algebra  {\it Weyl} is an associative algebra with generators
$ a^+(n), a^{-} (n) \quad (n \in {\Z})$ and relations
$$  [a^+(n), a^{-} (m)] = \delta_{n+m,0}, \quad [a^{\pm} (n), a ^{\pm}(m)]  =  0 \quad (n,m \in {\Z}). $$
Let $\mathcal W$ denote the simple {\it Weyl}--module generated by the cyclic vector ${\bf 1}$ such that
$$ a^+(n) {\bf 1} = a  ^- (n+1) {\bf 1} = 0 \quad (n \ge 0). $$
As a vector space $ W \cong {\C}[a(-n), a^*(-m) \ \vert \ n >0 , \ m \ge 0 ]. $
There is a  unique vertex algebra $(\mathcal W, Y, {\bf 1})$ where
the  vertex operator map is $ Y: \mathcal W  \rightarrow \mbox{End} (W )[[z, z ^{-1}]] $
such that
$$ Y (a^+(-1) {\bf 1}, z) = a^+(z), \quad Y(a^- (0) {\bf 1}, z) = a ^- (z),$$
$$ a^+ (z)   = \sum_{n \in {\Z} } a^+ (n) z^{-n-1}, \ \ a^{-}(z) =  \sum_{n \in {\Z} } a^{-}(n)
z^{-n}. $$
Vertex algebra $\mathcal W$ is called the Weyl vertex algebra or the $\beta \gamma $ vertex algebra.

\subsection{The vertex algebra $\Pi$}
\label{pi-voa}

Consider the rank two lattice $D = {\Z} c + {\Z} d $ such that $$\langle c , c \rangle = \langle d, d \rangle =0, \ \langle c, d \rangle =2. $$
Let $V_D = M_{c, d} (1) \otimes {\C} [D]$ be the associated vertex algebra, where   ${\C}[D]$ is the group algebra associated to the lattice $D$ and $M_{c,d}(1)$ is the Heisenberg vertex algebra generated with fields
$$ c(z) = \sum _{n \in {\Z} } c(n) z^{-n-1}, \ d(z) = \sum _{n \in {\Z} } d(n) z^{-n-1}, $$
with the commutation relations
$$ [c(n), c(m) ]= [d(n), d(m)] =0, \ [c(n), d(m) ] = 2 n \delta_{n+m, 0}. $$
The vertex algebra $\Pi$ is realized as a subalgebra of $V_D$:
$$ \Pi := M_{c,d} (1) \otimes {\C} [\Z c]. $$
This  vertex algebra was  introduced in \cite{DBT} and called the half-lattice vertex algebra.
Consider the  following irreducible, weight $\Pi$--module $\Pi_{r} (\lambda) = \Pi. e^{r \mu+ \lambda c}$, where $\mu = (c-d) /2$, $r \in {\Z}$, $\lambda \in {\C}$.
 Then we have the following well known  fusion rules (see \cite{DL, DBT,  AP, Wood-21}):
{$$ \Pi_{r_1} (\lambda_1)  \times \Pi_{r_{2} } (\lambda_2) = \Pi_{r_1 + r_2} (\lambda_1 + \lambda_2), \quad (r_1, r_2  \in {\Z}, \lambda_1, \lambda_2 \in {\C}). $$}

 \section{Heisenberg-Virasoro VOA of level zero and its fusion rules}
\label{Hvir-sect}

Recall that the twisted Heisenberg--Virasoro algebra is an infinite--dimensional complex Lie algebra $HVir$ with basis
$$
\{T_{HVir} (n),I_{HVir} (n):n\in\mathbb{Z}\}\cup\{C_{L},C_{LI},C_{I}\}
$$
and commutation relation:
\bea
&& \left[   T_{HVir} (n),T_{HVir}(m)\right]  =(n-m)T_{HVir} (n+m)+\delta_{n,-m}\frac{n^{3}-n}{12}%
C_{L},\\
&& \left[  T_{HVir}(n),I_{HVir} (m)\right]  =-mI_{HVir}(n+m)-\delta_{n,-m}( n^{2}+n)
C_{LI},\\
 && \left[  I_{HVir}(n),I_{HVir}(m)\right]  =n\delta_{n,-m}C_{I},\\ && \left[  {HVir} ,C_{L}\right]  =\left[  {HVir},C_{LI}\right]  =\left[ {HVir} ,C_{I}\right]  =0.
\eea

Let the central elements act as identity: $C_L = c_L Id, C_I = c_I Id, C_{LI} = c_{LI} Id, c_L,c_I,c_{LI} \in {\C}$.
Let $V^{HVir}(c_{L},c_{I},c_{L,I},h_I,h)$ denote the Verma module with highest
weight $(c_{L},c_{I},c_{L,I},h_I,h)$, and  $L^{HVir}(c_{L},c_{I},c_{L,I},h,h_{I})$ its irreducible quotient (cf.\ \cite{Billig}), where $ h_I, h \in {\C}$ are highest weights with respect to the action of $I_{HVir}(0),T_{HVir}(0)$, respectively.
Let $V^{HVir} (c_L, c_{I}, c_{L, I}) $  be   the universal vertex algebra associated to  $HVir$. If $c_{I} = 0, c_{L, I} \ne 0$, then $V^{HVir} (c_L, 0, c_{L, I})$ is simple, and we set $L^{HVir} (c_L, c_{L,I}) = V^{HVir} (c_L, 0, c_{L, I})$.

In this paper we consider the case:

$$ c_L=2, c_{L,I} = 1; c_I=0. $$ For simplicity we shall denote the Verma module $V^{HVir}(c_{L}, 0,c_{L,I},h_I,h)$ with $V^{HVir} [h_{I}, h]$ and its irreducible quotient by $L^{HVir}[h_{I}, h]$.
Then $L^{HVir}:= L^{HVir}[0,0]$ is a simple vertex operator algebra (cf. \cite{AR, Billig}).

\begin{remark} The  general Heisenberg-Virasoro vertex algebra  with $c_I =0$, and $c_{L,I} \ne 0$ is considered in   \cite{AR, Billig}. But it is always isomorphic to $L^{HVir}$.
More precisely, let $L, I$ be generators of $L^{HVir}$, and let $\bar L, \bar I$ be generators of  $L^{HVir} (c_L, c_{L,I})$. One shows that
 \bea
&&f: L^{HVir} (c_L, c_{L,I})\rightarrow  L^{HVir} \nonumber \\
&& \bar L \mapsto  L - \frac{c_L-2}{24} \partial  I, \ \
\bar  I \mapsto  c_{L,I}  I \nonumber
\eea
is an isomorphism of vertex algebras.
\end{remark}

Let $L^{HVir} [x,y]$ denote the irreducible, highest weight $L^{HVir}$--module with highest weight $(x,y)$ and highest weight vector $v_{x,y}$ such that
$$ I_{HVir} (0) v_{x,y} = x v_{x,y}, \quad  T_{HVir} (0) v_{x,y} = y v_{x,y}. $$

The characters of irreducible, highest weight modules were obtained in  \cite{Billig}:
\begin{itemize}
\item If $x \in {\C} \setminus {\Z}$, or $x=1$, then $V^{HVir}[x,y]$ is irreducible and  $$\mbox{ch}_q L^{HVir} [x,y] =  q^{y-1/12}\prod_{j=1} ^{\infty} (1-q^j) ^{-2}.$$
\item If  $x \in {\Z} \setminus \{1\}$, set $ \vert x - 1\vert = p $. Then $V^{HVir}[x,y]$ possesses a singular vector of degree $p$ which generates the maximal submodule and   $$\mbox{ch}_q L^{HVir} [x,y] =q^{y- 1/12}  (1 - q^p) \prod_{j=1} ^{\infty} (1-q^j) ^{-2}. $$

\end{itemize}

Let  $\Delta_{r,s} =  (r+1)s$. We shall need the following  fusion rules results from \cite{AR}:
\bea L^{HVir} [-1,0] \times L^{HVir}[ -r, \Delta_{r,s}] &=&  L^{HVir} [-r-1, \Delta_{r+1, s}] \ \mbox{for}\  r \notin   \Z_{< 0}, \nonumber \eea
Since $\Delta_{r+1, s} = \Delta_{r,s} +  s$ we get fusion rules
$$   L^{HVir} [-1,0] \times L^{HVir}[ -r, y ]  = L^{HVir} [-r-1, \frac{r+2}{r+1} y], \quad  y = \Delta_{r,s}, \ r \notin   \Z_{< 0}. $$
 Since $( L^{HVir} [x,y])^{*} = L^{HVir} [2-x, y]$, using the adjoint intertwining operator we get the fusion rules:
$$   L^{HVir} [-1,0] \times L^{HVir} [r+3, \frac{r+2}{r+1} y]  =  L^{HVir}[ r+2, y ]   \quad  y = \Delta_{r,s}, \ r \notin   \Z_{< 0}. $$

\section{ Affine Nappi-Witten vertex operator  algebra }
\label{affine-nw-voa}

In this section we review basic definitions and properties of the  affine Nappi-Witten vertex algebra  following papers \cite{BJP, BKRS}. The terminology and results used  are mostly  from \cite{BKRS}. In the following section we shall interpret these results via inverse QHR.

\vskip 5mm
Let $\mathfrak h_4$ be  the four-dimensional complex Lie algebra with basis $\{E,F, I,J\}$ whose nonzero Lie
brackets are, modulo antisymmetry, as follows:
$$  [E,F] = I, \  [J,E] = E,\  [J,F] = -F. $$
Note that $I$ is in the center of $\mathfrak h_4$.

As in \cite{BKRS}, we fix the symmetric non-degenarate invariant  bilinear form:
$$ ( E, F ) =  (I, J ) = 1,   (I, I ) = 0, $$
and all other products are zero.

Lie algebra $\mathfrak h_4$  has the following triangular decomposition
$\mathfrak h_4 = {\mathfrak h_4} ^+ \oplus  {\mathfrak h_4} ^0 \oplus  {\mathfrak h_4} ^-, $
where
$$ {\mathfrak h_4} ^+  = {\C} E, \  {\mathfrak h_4} ^-  = {\C} F,  \ {\mathfrak h_4} ^0 = {\C} I + {\C} J. $$
Define  $\mathfrak b ^{\pm} =  {\mathfrak h_4} ^0  + {\mathfrak h_4} ^{\pm}$.
For any $(i,j) \in {\C}$, define
$$ \mathcal{V}_{i,j}^{\pm} = U(\mathfrak h_4) \oplus_{ U(\mathfrak b^{\pm})} {\C} v_{i,j}, $$
where $  {\C} v_{i,j} $ is a $\mathfrak b^{\pm} $--modules such that $ I v_{i,j} = i v_{i,j}$, $J v_{i,j} = j v_{i,j}$, and
$ {\mathfrak h_4} ^{\pm}$ acts as $0$.  Module $\mathcal V_{i,j} ^+ $ is called the (highest weight) Verma $\mathfrak h_4$--module, and  $\mathcal V_{i,j} ^-$ the lowest weight Verma module.

The Verma module  $\mathcal V_{i,j} ^+$ is irreducible for  $i \ne 0$. For $i = 0$, we have  a family of irreducible $1$--dimensional $\mathfrak h_4$--modules realized as follows:
$$  \mathcal L_{0,j} = \mathcal V_{0,j} ^+ / \mathcal V_{0,j-1}^+. $$

Lie algebra $\mathfrak h_4$ also has a family of irreducible modules with $1$-dimensional weight spaces which are neither highest nor lowest weight modules.  These modules were described in \cite{BKRS}.  Set  $ \Omega = F E  + IJ$. Then $\Omega $ is a central element of $U(\mathfrak h_4)$. Consider the irreducible $1$--dimensional ${\C}[I, J, \Omega]$--module ${\C} w_{i,j, h}$ such that $$I w_{i,j, h} = i w_{i,j,h}, \ J w_{i,j, h} = j w_{i,j,h}, \  \Omega  w_{i,j,h} = h w_{i,j, h}.$$
Then we have  the induced module
$ \mathcal R_{i,j;h} = \mbox{Ind} _{\C[I,J,\Omega ]} ^{U(\mathfrak h_4)} {\C} w_{i,j, h}$.
It was proved in \cite[Proposition 2]{BKRS} that $ \mathcal R_{i,[j];h}$ is irreducible iff $i,h \in {\C}$,  $[j] \in {\C} / {\Z}$, $h \notin i [j]$.
If $i\ne 0$, $ h \in i [j]$, then  $ \mathcal R_{i,[j];h}$ is reducible, and it can contain a unique highest weight submodule or a lowest weight submodule; we shall denote the former with
$ \mathcal R^+_{i, h}$, and the latter with  $ \mathcal R^-_{i,h}$.

Let $\widehat{\mathfrak h}_4= \mathfrak h_4 \otimes {\C}[t, t^{-1} ] \oplus {\C} K$
be the associated affine Lie algebras with the commutation relation
$$ [x(n), y(m)] = [x,y](n+m) + n { (x , y) }\delta_{n+m,0} K, $$
where  we set $x(n) = x \otimes t^n$ and $K$ is central element. We will identify $\mathfrak h_4$ with $\mathfrak h_4 \otimes t^0$.

 Then   there is a universal affine vertex algebra $V^k(\mathfrak h_4)$ of level $k$ which is generated by the fields (cf. \cite{Kac, LL, BKRS})
$$ x(z) = \sum _{n\in {\Z}} x(n) z^{-n-1}, \quad x \in \mathfrak h_4. $$

Consider the following triangular decomposition  $\widehat{\mathfrak h}_4=\widehat{\mathfrak h_4}^+ \oplus \widehat{\mathfrak h_4}^0 \oplus \widehat{\mathfrak h_4}^-$, where
\begin{eqnarray}
&& \widehat{\mathfrak h}^{\pm} ={\mathfrak h} \otimes  t^{\pm 1} {\C}[t^{\pm 1}], \quad
 (\widehat{\mathfrak h}_4) ^0 = \mathfrak h_4 + {\C} K. \label{trdec}
\end{eqnarray}

Let $k \in {\C}$ and $U$ be any $\mathfrak h_4$--module. Then we can consider $U$ as $\widehat{\mathfrak p} = \mathfrak  h_4 +   \widehat{\mathfrak h}^{+} $--module with $\widehat{\mathfrak h}^{+}$ acting trivially and $K$ acting as $k \mbox{Id}$. Then we have the induced module $$V_{\widehat{\mathfrak h}_4} (k, U)  = U(\widehat{\mathfrak h}_4) \otimes _{U(\widehat{\mathfrak p})} U.$$

Using the fact that any $\widehat{\mathfrak h}_4$ module with $K = k {\bf{1}}, k\in \mathbb{C}\setminus \{0\}$ is isomorphic to the module with $k=1$, we shall fix this $k$ value in what follows and consider the algebra $V^1 (\mathfrak h_4)$.  We write
$M(U) =  V_{\widehat{\mathfrak h}_4} (1, U)$.  Top component (also called the  the space of ground states) is then isomorphic to $U$.
Let $J(U)$ be the sum of all $V^1 (\mathfrak h_4)$--submodules of $M(U)$   that have zero intersection with the top component.  Define the quotient module
$$ L(U) = M(U) / J(U). $$
If $U$ is irreducible $\mathfrak h_4$--module, then $L(U)$ is irreducible $V^1 (\mathfrak h_4)$--module. Note that even if $U$ is not irreducible, then $L(U)$ is also well-defined, but reducible.

Now we have the following  $\widehat{\mathfrak h}_4$--modules (which are also   $V^1 (\mathfrak h_4)$--modules):
\begin{itemize}
\item Verma modules $\widehat{\mathcal V } ^{\pm} _{i,j} = M( \mathcal V_{i,j}) $,  their irreducible quotients  $\widehat{\mathcal L } ^{\pm} _{i,j} = L( \mathcal L^{\pm} _{i,j}) $, {$i\in\mathbb C\setminus 0$}.
\item Modules $\widehat{ \mathcal L}_{0,j} = L( \mathcal L_{0,j})$ whose top components are $1$--dimensional.
\item Relaxed modules $\widehat {\mathcal R}_{i, [j], h} = M (\mathcal R_{i, [j], h})$, and their quotients   $\widehat {\mathcal E}_{i, [j], h} = L (\mathcal R_{i, [j], h})$.
\item Reducible relaxed modules  $\widehat {\mathcal R}^{\pm}_{i, , h}$, and their (also reducible) quotients   $\widehat {\mathcal E}^{\pm} _{i, , h} = L (\mathcal R^{\pm} _{i, h})$.
\end{itemize}

Note that $\widehat{\mathcal{L}}_{0,0}= \widehat{\mathcal{L}}_{0,0}^+ \cong \widehat{\mathcal{L}}_{0,0}^- = V^1(\mathfrak h_4)$.   It was proved in \cite[Prop. 4]{BKRS} that $V^1 (\mathfrak h_4)$ is a simple vertex operator algebra and   has the conformal vector
 $$ \omega  ^{sug}  _{\mathfrak h_4} =   \left( E(-1)  F(-1)  +  I(-1)  J(-1)  -\frac{1}{2} I(-2)   - \frac{1}{2} I(-1) ^2 \right){\bf 1}.$$
{We set
$$ L^{sug} _{\mathfrak h_4} (z) = Y( \omega  ^{sug}  _{\mathfrak h_4}, z) = \sum _{n \in {\Z}}  L^{sug}_ {\mathfrak  h_4} (n) z^{-n-2}. $$ 
}




   We recall the following important result from \cite{BKRS}.
\begin{theorem} \label{main-bkrs} \cite[Theorems  6 and 8]{BKRS}
  Every irreducible relaxed highest weight  $V^1 (\mathfrak h_4)$-module 
     is isomorphic to one, and only one, of the following modules:
\begin{itemize}

		\item $\widehat{\mathcal L}_{0,j}$, with $j\in\mathbb C$.
		\item $\widehat{\mathcal L}_{i,j}^\pm$, with $i\in\mathbb C\setminus 0$ and $j\in\mathbb C$.
		\item $\widehat{\mathcal E}_{i,[j],h}$, with $i\in\mathbb C\setminus 0$, $[j]\in\mathbb C/\mathbb Z$, $h\in\mathbb C$ and $h/i \notin [j]$.
		\item $\widehat{\mathcal R}_{0,[j],h}$, with $[j]\in\mathbb C/\mathbb Z$ and $h\in\mathbb C\setminus 0$.
	\end{itemize}
 Module $\widehat {\mathcal E}^{\pm} _{i, , h} $ for $i \ne 0,  -1$,  is of length two and we have the following
nonsplit short exact sequences:  	
	\bea
	0 \rightarrow  \widehat{\mathcal L}_{i,h/i}^+ \rightarrow  \widehat {\mathcal E}^{+ } _{i, , h} \rightarrow  \widehat{\mathcal L}_{i,h/i+1}^- \rightarrow 0,  \quad \
	0 \rightarrow  \widehat{\mathcal L}_{i,h/i+1}^- \rightarrow  \widehat {\mathcal E}^{- } _{i, , h} \rightarrow  \widehat{\mathcal L}_{i,h/i}^+ \rightarrow 0.\nonumber
	\eea
\end{theorem}

One can also consider the spectral  flow automorphism $\sigma^{\ell}$, $\ell \in {\Z}$, given by
\bea
\sigma^{\ell}(E(n) )= E(n-\ell),   \sigma^{\ell}(I(n))  = I(n)- \ell \delta_{n,0} K, \sigma^{\ell} (J(n)) = J(n), \sigma^{\ell} (F(n)) = F(n+\ell), \sigma^{\ell} (K) = K; \nonumber
\eea
and automorphism $s_{t}$, $t \in {\C}$, given by

 \bea
s_t(E(n) )= E(n),  s_t(I(n)) = I(n),   s_t (J(n)) = J(n)  - t \delta_{n,0} K, s_t(F(n)) = F(n), s_t (K) = K. \nonumber
\eea

Note that  the automorphisms $\sigma^{\ell}$ and $s_t$ commute. We define the automorphism $g= \sigma \circ s_{t=-1/2}$.

Now let $h = \frac{I}{2} - J$. Then $h (1) h = - {\bf 1}$. Let $$\Delta(h, z) = z^{h(0)} \exp \sum_{n=1} ^{\infty} \left( \frac{h(n) }{-n} (-z)^{-n}\right). $$
For any $V^1 (\mathfrak h_4)$--module $(M, Y_M)$, we have a new $V^1 (\mathfrak h_4)$--module
$$ (\rho_{\ell} (M), Y_{\rho_{\ell} (M)} (\cdot, z)) :=(M, Y_M(\Delta(\ell h, z)\cdot, z)),$$
with the action of the generators given by
\bea   Y_{\rho_{\ell} (M)} (E, z) &=& z^{-\ell} Y_M (E, z) = z^{-\ell} E(z), \nonumber \\
 Y_{\rho_{\ell} (M)} (F, z)&=& z^{\ell} Y_M (F, z) = z^{\ell} F(z), \nonumber \\
  Y_{\rho_{\ell} (M)} (I, z) &=&   Y_M (I,  z) -z^{-1} \ell \mbox{Id}  = I(z)- z^{-1} \ell \mbox{Id}  \nonumber  \\
   Y_{\rho_{\ell} (M)} (J, z) &=&   Y_M (J,  z) + z^{-1}\frac{\ell}{2} \mbox{Id} = J(z) + z^{-1}\frac{\ell}{2} \mbox{Id}   \nonumber
 \eea
 One sees that  the $\rho_{\ell} (M)$ is obtained from $M$ by applying the automorphism $g^{\ell}$.
 Note also that $g (V^{1} (\mathfrak h_4) ) =  \widehat{\mathcal  L}^- _{-1, \frac{1}{2}}$.

   \section{Quantum Hamiltonian Reduction}
   In this section we shall  show that   the Quantum Hamiltonian Reduction (QHR) of $V^1 (\mathfrak h_4)$ is the Heisenberg-Virasoro algebra of level zero $L^{HVir}$. Our approach will use a version of the  Kac-Wakimoto-Roan construction  of the affine  $W$-algebra  $W^k(\g, x, f)$.

   \vskip 5mm

  For any simple Lie algebra $\g$   and pair $(x,f)$ such that $ad(x)$  is diagonalizable on $\g$,  $[x,f] = -f$ and $f$ is an even nilpotent element, Kac,  Wakimoto and Roan  in \cite{KWR, KW04} constructed the universal affine $W$--algebra $W^k(\g, x, f)$. They proved that $W^k(\g, x, f)$ are strongly generated by fields
  $J^{a_i}$, where $a_1, \dots, a_s$ is a basis
of the centralizer $g^{f}$ consisting of eigenvectors of $ad(x)$.

$W^k(\g, x, f)$ was constructed as the cohomology of the complex
$$ \mathcal C(\g,x,f,k) = (V^k(\g) \otimes F^{ch} \otimes F^{ne}, d_0),$$
where $V^k(\g)$ is the universal affine vertex algebra of level $k$ associated to $\g$, $F^{ch}$ is the free charged  fermion vertex algebra associated to $\g_+ + {(\g_+)^*}$, {where $(\g_+)^*$ is dual to $\g_+$}, $F^{ne}$ is the neutral free fermion vertex algebra based on $\g_{1/2}$, and  $d_0$ is an
 odd derivation of the vertex algebra $\mathcal C(\g,x,f,k) $ whose square is $0$ (see \cite{KW04} for the definitions).

 We want to extend their approach to  the  Nappi-Witten Lie algebra $\g= \mathfrak h_4$ and construct the complex    $\mathcal C(\g,x,f,k)$. Let us put here $k=1$, since all Nappi-Witten vertex algebras are isomorphic to $V^1 (\mathfrak h_4)$.  Although  the  Nappi-Witten Lie algebra $\g=\mathfrak h_4$ is not simple, we can find a pair $(x,f)$ for which the Kac-Wakimoto theorem still works. Let us take $x=J, f = F$. Then we have the gradation
\bea \g &=& \g_{-1} + \g_0 + \g_1, \ \g_{-1} = {\C} F, \ \g_1 = {\C} E, \ \g_0 = \mbox{span}_{\C} \{I, J\}, \nonumber \\
\g^f &= & \g_0 ^{f} + \g_{-1} ^{f}, \ \  \g_0 ^{f}  = {\C} I, \ \g_{-1} ^f = {\C} F,\nonumber  \\
\g_+ &= &{\C} E,  \ \ \g_- = {\C} F. \nonumber \eea
Note that $\g_{1/2} =\{0\}$. In our case   the complex is just
  $$ \mathcal C(\g,x,f) = (
  V^1 (\mathfrak h_4) \otimes F^{ch}  , d_0),$$
where
  $\mathcal F^{ch} = \mathcal F $ be a the  fermionic vertex superalgebra generated by the odd fields $\Psi^{+} (z) = \sum _{n \in {\Z} } \Psi^{+} (n) z^{-n-1}$,
 $\Psi^{-} (z) = \sum _{n \in {\Z} } \Psi^{-} (n) z^{-n}$   such that the components of the fields satisfies the anti--commutation relation for the Clifford algebra at level $K=1$:
 $$[\Psi^{\pm} (m), \Psi^{\pm}(n)]_+ =0, \quad [\Psi^{+} (m), \Psi^{-}(n)]_+ = \delta_{n+m, 0} K, \quad (m, n  \in   {\Z}).$$
 The conformal vector  of central charge $c=-2$ is $\omega_{fer} = \Psi^- (-1) \Psi^+ (-1) {\bf 1}$. So $\mathcal F$ becomes a ${\Z}_{\ge 0}$--graded vertex operator superalgebra of central charge $c=-2$.
  By using the boson-fermion correspodence, $\mathcal F= V_{\Z\alpha} =M_{\alpha} (1) \otimes {\C}[\Z \alpha]$, where  $\langle \alpha, \alpha \rangle  =1$,  $M_{\alpha} (1)$ is the Heisenberg vertex algebra generated by $\alpha(z)$, $\Psi ^{\pm} = e^{\mp \alpha}$.  The conformal vector is given by
 $$ \omega_{fer} = \frac{1}{2}\left(  \alpha(-1) ^2 - \alpha(-2) \right) {\bf 1}. $$

 The vertex superalgebra $\mathcal F$ has the charge decomposition $\mathcal F= \bigoplus_{\ell \in {\Z}} \mathcal F^{(\ell)}$ such that
 $ \mathcal F^{(\ell) } = \{ v \in \mathcal F \ \vert \ \alpha(0) v = - \ell v \}$.

 Next, we need to determine the derivation $d_0$ on the complex.
  Let  $$ d_{DS} = {(E(-1){\bf 1} + {\bf 1})} \otimes e^{\alpha}. $$ Then $d_0 =  \mbox{Res}_z d_{DS}(z) $ is the derivation of our complex  $\mathcal C(\g,x,f)$. As in \cite{KW04},  we  define the  vertex algebra
 $W(\g, x, f)$  as the zero cohomology of the complex  $\mathcal C(\g,x,f)$, so  $W(\g, x, f) =  H^0(\mathcal C(\g,x,f))$.  

 \begin{proposition}  \label{prop-ds}  In the above notation,
 $W(\mathfrak h_4, x, f) =   L^{HVir}$.
 \end{proposition}
 \begin{proof}
 { One can check that     $H^i(\mathcal C(\g,x,f)) = 0$ for $i \ne 0$  following proof of  \cite[Theorem 4.1]{KW04}. }
Then, following  \cite{KW04}   (see also  \cite[Theorem 4.1]{AMP}) one can show that there is a graded vector space isomorphism
 $$ { \Phi}: S(\widehat{\g^f} ) \cong  W(\g, x, f) $$
 where
 $$  \widehat{\g^f}  = \g_{-1} ^f  \otimes  t^{-2} {\C}[t^{-1}]  \oplus
\g_0 ^{f} \otimes  t^{-1} {\C}[t^{-1}], $$
 such that  $ W(\g, x, f)$ is freely generated by $\hat L = { \Phi} (F)$  and $ \hat I = { \Phi} (I)$.
  These fields are represented by
 $$\hat I  = I, \ \  \hat L  =      \omega ^{sug}  _{\mathfrak h_4} + \partial J + \omega_{fer}. $$

  One  sees  that  $\hat L $ is the Virasoro field of central charge $c_L = 2$,  $ \hat I  $  is the commutative Heisenberg field and that
 $$\hat  L(-1) \hat I  = \partial  \hat I, \hat L (0) \hat  I = I, \hat L(1)  \hat I = -2 c_{\hat L \hat I} {\bf 1}, $$
 for $c_{\hat L \hat I}= 1$. Therefore $\hat L, \hat I$ generate the Heisenberg Virasoro vertex algebra   $L^{HVir}$.  The proof follows. 
 \end{proof}

\section{Realisation of the affine Nappi-Witten vertex operator  algebra as QHR inverse of Heisenberg-Virasoro algebra at level zero }

In this section we shall prove that there is an embedding $V^1 (\mathfrak h_4) \hookrightarrow L^{HVir} \otimes \Pi$ which inverts the QHR from the previous section.   Our methods follow a similar approach to the one used in the realization of $V^k(\mathfrak{sl}_2)$ from \cite{A-2017}.
 Using the  Wakimoto realization of $V^1 (\mathfrak h_4) $ from  \cite{BJP} and the   bosonization of  the $\beta \gamma$ vertex algebra, we construct  a new realization of $V^1 (\mathfrak h_4) $ together with a corresponding screening operator. \vskip 5mm

Let $L = {\Z} \alpha + \Z {\beta}$ be the integral lattice with the following scalar products: $$\langle \alpha, \alpha \rangle  = - \langle \beta, \beta  \rangle =1, \quad  \langle \alpha, \beta \rangle = 0. $$
Let $V_L= M_{\alpha, \beta} (1) \otimes {\C}[L]$ be the associated lattice vertex algebra,  where $M_{\alpha, \beta} (1)$ is the Heisenberg vertex algebra generated by the fields
$$ \alpha(z) = \sum _{n \in {\Z}} \alpha(n) z^{-n-1}, \  \beta(z) = \sum _{n \in {\Z}} \beta(n) z^{-n-1}, $$
satisfying the following commutation relations:
$$ [\alpha(n), \alpha(m)] = n \delta_{n+m,0},  \   [\beta(n), \beta(m)] = -n \delta_{n+m,0},\  [\alpha(n), \beta(m)] = 0  \quad (n,m \in {\Z}), $$
and ${\C}[L]$ is the group algebra associated to the lattice $L$.

Then there is an embedding $  \mathcal W \hookrightarrow \Pi_{\alpha, \beta} \subset  V_L$
 $$a^+ \mapsto  e^{\alpha + \beta}, a^- \mapsto  - \alpha (-1) e^{-\alpha - \beta}, $$
where  $\mathcal W$ is the Weyl vertex algebra  (cf. Subsection \ref{sub-sect-weyl}), and $\Pi_{\alpha, \beta} = M_{\alpha, \beta} (1) \otimes {\C}[\Z (\alpha+ \beta)]  \cong \Pi $  (cf. Section  \ref{pi-voa}).  
{
 The screening operator for this embedding is $S = e^{\alpha} _0 = \mbox{Res}_zY(e^{\alpha}, z) $
 and  $ \mathcal W \cong \mbox{Ker}_{ \Pi_{\alpha, \beta} } S$. 
 }

Let $M_{p,q}(1)$ be the rank two Heisenberg vertex algebra generated by fields
$$ p(z) = \sum _{n \in {\Z}} p(n) z^{-n-1}, \  q(z) = \sum _{n \in {\Z}} q(n) z^{-n-1}, $$
satisfying the following commutation relations:
$$ [p(n), q(m)] = n \delta_{n+m,0},  \ [p(n), p(m)] = [q(n), q(m)] = 0, \quad (n,m \in {\Z}). $$

 The following result was proved in \cite{BJP}:
 \begin{theorem} \cite[Theorem 5.1]{BJP} \label{wak-nw}
There exists a vertex algebra homomorphism   ${\rho_1}: V^1 (\mathfrak h_4) \rightarrow  \mathcal W \otimes M_{p,q}(1)$, uniquely determined by
\bea
E &\mapsto & a^+, \nonumber \\
F& \mapsto &  \partial a^- + p(-1) a^-,  \nonumber \\
I&\mapsto & p(-1){\bf 1}, \nonumber \\
J&\mapsto & \frac{1}{2} p(-1){\bf 1} + q(-1){\bf 1} - a^+(-1) a^- .\nonumber
\eea
 \end{theorem}

By reorganising the appropriate fields, we shall now see that ${ \mbox{Im} (\rho_1)} \subset L^{HVir} \otimes \Pi$.  (Recall that the algebra $\Pi$  was defined in Subsection \ref{pi-voa}.)

\begin{theorem} \label{inverse-realization} There is an embedding of ${ \mathcal \rho} : V^1 (\mathfrak h_4) \rightarrow L^{HVir} \otimes \Pi$ uniquely determined by:
 \bea
E &\mapsto &  e^{c },   \label{real-inverse-1}
 \\
F& \mapsto &  \left(   T_{HVir} (-2)    - \nu(-1) I_{HVir}   -\nu(-2)   \right)e^{-c},  \label{real-inverse-2}
 \\
I&\mapsto &c(-1){\bf 1}+ I_{HVir},  \label{real-inverse-3}
\\
J&\mapsto & \frac{1}{2} d(-1) {\bf 1}+\frac{1}{2} I_{HVir}, \label{real-inverse-4}
\eea
where $\nu = \frac{c+d}{2} = \alpha + q$.

The Sugawara Virasoro  vector   $\omega ^{sug} _{\mathfrak h_4}$  is { mapped to}
\bea \rho(\omega ^{sug} _{\mathfrak h_4}) = T_{HVir}   - \frac{1}{2}  I_{HVir} (-2) {\bf 1}+  \frac{1}{2}   c(-1) d(-1){\bf 1}   -\frac{1}{2}  (c(-2)+ d(-2)){\bf 1}. \label{expr-sug-1} \eea
 The screening   operator for this realization is
$$ S =  s_0 = \mbox{Res}_z Y(s, z) : L^{HVir} \otimes \Pi \rightarrow L^{HVir}[-1,0] \otimes \Pi_{-1}(0), $$where $s = v_{-1,0} \otimes e^{\nu} $, 
$v_{-1,0}$ is the highest weight vector of $L^{HVir}[-1,0]$.
Moreover, 
\bea V^1 (\mathfrak h_4) { \cong} Ker_{ L^{HVir} \otimes \Pi} (S).  \label{desc-ker} \eea

\end{theorem}
\begin{proof}
First we consider homomorphism $\rho_1 : V^ 1 (\mathfrak h_4)\rightarrow  \mathcal W \otimes M_{p, q}(1)  \subset \Pi_{\alpha, \beta} \otimes M_{p,q}(1) $  from Theorem    \ref{wak-nw}.
 We have
\bea
\rho_1 (E) &=& a^+(-1) {\bf 1}    = e^{\alpha + \beta}   \nonumber \\
\rho_1(F)& =&    a^-(0) {\bf 1}   + p(-1)  a^-(0)  {\bf 1} \nonumber \\
&=& \partial  (-\alpha(-1) e^{-\alpha -\beta} ) - p(-1) \alpha(-1) e^{-\alpha-\beta} \nonumber \\
&=&\left(  \alpha(-1) ( \alpha(-1) + \beta(-1)- p(-1) ) - \alpha(-2)  \right) e^{-\alpha-\beta} \nonumber \\
\rho_1(I)&=& p(-1) {\bf 1}\nonumber \\
\rho_1(J)&=& \frac{1}{2} p(-1){\bf 1} + q(-1){\bf 1} - a^+(-1) a^- \nonumber \\
&=&\left( \frac{1}{2} p(-1) + q(-1) - \beta(-1)\right) {\bf 1}. \nonumber
\eea

Set
$$ c= \alpha+ \beta , \ d= \alpha- \beta + 2 q,$$
and let $M_{c,d}(1)$ be the Heisenberg vertex algebra generated by $c,d$ and
$$\Pi = M_{c,d} (1) \otimes {\C}[\Z c]$$
be the half-lattice vertex algebra as in Subsection \ref{pi-voa}.
Then $ \Pi ^{\perp} = M_{c_1, d_1}$, where $M_{c_1, d_1}(1)$ be the Heisenberg vertex algebra generated by
$$c_1 =   \alpha + \beta - p,  \  d_1= -2 q, \  2 \beta = c-d   -d_1.  $$
So ${ V^1 (\mathfrak h_4)   \cong \mbox{Im} (\rho_1)} \subset M_{c_1, d_1} (1) \otimes \Pi$. Next we notice
$$ \alpha = (c + d + d_1) / 2, \ \nu = (c+d)/ 2, \ \mu = (c-d) /2 . $$

 Let  $L^{HVir}$ be the  simple Heisenberg-Virasoro vertex algebra as in Section \ref{Hvir-sect}.
This vertex algebra is realized as a subalgebra of the Heisenberg vertex algebra $M_{c_1, d_1}(1)$ (cf. \cite{AR}) { generated by:}
$$ T_{HVir} 
 = (\tfrac{1}{2} c_1(-1) d_1(-1)  - \tfrac{d_1(-2)}{2} ){\bf 1},  \ I_{HVir} = - c_1 (-1) {\bf 1}. $$
  By \cite{AR}, $L^{HVir}.  e^{\frac{d_1}{2}} = L^{HVir}[-1,0]$, so we can set  
  $v_{-1,0} = e^{\frac{d_1}{2}}$. Then $ s = e^{\alpha} = v_{-1,0}  \otimes e^{\nu}$.  

Moreover, by \cite[Corollary 2.9]{AR} we have 
$$ L^{HVir} = \mbox{Ker} _{M_{c_1, d_1}(1)} Q, \quad  Q =  \mbox{Res}_z e^{c_1} (z) = e^{c_1} _0. $$


Since $e^{c_1} _0$ acts trivially on $V^1 (\mathfrak h_4)   $, we get
$$  {V^1 (\mathfrak h_4)   \cong \mbox{Im} (\rho_1)}     \subset L^{HVir}  \otimes \Pi. $$

 We get:

 \bea
\rho_1 (E)  &=&  e^{c},  \nonumber \\
\rho_1 (F)& =&  \left(   T_{HVir} (-2)    + \nu(-1) c_1(-1)   -\nu(-2) - \tfrac{d_1(-2) }{2}\right)e^{-c}, \nonumber \\
 & =&  \left(   T_{HVir} (-2)    - \nu(-1) I_{HVir}   -\nu(-2)   \right)e^{-c},  \nonumber \\
\rho_1(I)&=&c(-1){\bf 1}+ I_{HVir}, \nonumber \\
\rho_1(J)&=& \frac{1}{2} d(-1){\bf 1} +\frac{1}{2} I_{HVir}.\nonumber
\eea
This proves  that the formulas  (\ref{real-inverse-1})- (\ref{real-inverse-4})   uniquely determine the homomorphism  ${ \mathcal \rho} : V^1 (\mathfrak h_4) \rightarrow L^{HVir} \otimes \Pi$ such that $\mbox{Im} (\rho) = \mbox{Im} (\rho_1)$.

Next we notice that
\bea
\rho(E(-1)   F(-1) {\bf 1}) &=&      T_{HVir}     - \nu(-1) I_{HVir}    -  \nu(-2){\bf 1}  + c(-1)   I_{HVir}+ \frac{1}{2} (  c(-1) ^2  + c(-2) ){\bf 1}, \nonumber  \\
\rho(I(-1)  J(-1){\bf 1}) &=& \frac{1}{2} c(-1)  d(-1) {\bf 1} + \frac{I_{HVir}(-1) }{2} (c(-1) + d(-1)){\bf 1} + \frac{1}{2} I_{HVir} (-1) ^2{\bf 1}  \nonumber \\
&=&  \frac{1}{2}  c(-1)  d(-1) {\bf 1} + \nu(-1)  I_{HVir}  + \frac{1}{2} I_{HVir} (-1)  ^2 {\bf 1} \nonumber \\
-  \frac{1}{2}\rho(I(-2){\bf 1})&=& -\frac{  c(-2) {\bf 1}  }{2} - \frac{\partial I_{HVir}}{2},  \nonumber \\
- \rho(I(-1) ^2 {\bf 1}) &=& - \frac{1}{2} I_{HVir}  (-1) ^2 {\bf 1}  - \frac{1}{2} c(-1) ^2 {\bf 1}  - c(-1)  I_{HVir}, \nonumber
\eea
which implies
\bea
 \rho(\omega ^{sug} _{\mathfrak h_4}) &=& T_{HVir}   - \frac{\partial I_{HVir}}{2} +  \frac{1}{2}   c(-1)  d (-1) {\bf 1}   -\frac{1}{2}   (c(-2) + d(-2)){\bf 1} .
 \nonumber \\
 &=& T_{HVir}    +   \frac{1}{2}  c(-1)  d(-1) {\bf 1}   -\frac{1}{2}     d(-2) {\bf 1}  - \frac{\partial  I }{2}. \nonumber \eea

Note that $S$ is screening operator for the embedding of the $\beta \gamma$ vertex algebra  $\mathcal W \hookrightarrow  \Pi_{\alpha, \beta}$, from which it follows  that $S$ commutes with the   $\beta \gamma$ generators.  This implies that
$$[S, \rho(E)] = [S, \rho(F)] = [S, \rho(I)] = [S, \rho(J)] =0. $$
Hence $S$ commutes  with the action of $V^1(\mathfrak h_4)$. Therefore, $V^1 (\mathfrak h_4) \subset \mbox{Ker}_{ L^{HVir} \otimes \Pi}  (S)$.

Now we consider  $L^{HVir} \otimes \Pi$ as a  $V^1(\mathfrak h_4)$--module. We shall see below in Corollary \ref{str-voa-relaxed} that $L^{HVir} \otimes \Pi$ has length two, and  that it is generated by the
highest weight vector ${\bf 1} $ and the subsingular vector $e^{-\alpha - \beta }$. Since
$$ S. e^{-\alpha- \beta } = e^{\alpha} _0 e^{-\alpha - \beta } = e^{-\beta} \ne 0, $$
we get that  $\mbox{Ker}_{ L^{HVir} \otimes \Pi}  (S)$ is generated (as  a $V^1(\mathfrak h_4)$--module) by ${\bf 1}$. Hence    (\ref{desc-ker})  holds. This finishes the proof of the theorem.
\end{proof}


\section{Realization of relaxed highest weight modules}
We shall consider  the weight $\Pi$--module $\Pi_{r} (\lambda) = \Pi e^{r \mu+ \lambda c}$, where $\mu = (c-d) /2$, $r \in {\Z}$, $\lambda \in {\C}$ (cf. Section \ref{pi-voa}).

Irreducible relaxed highest weight $V^1 (\mathfrak h_4)$--modules were introduced and classified  in \cite{BKRS}. In this section we will prove that those modules which are not of highest/lowest weight type can be realized as
$ L^{HVir} [x,y] \otimes \Pi_{1}(\lambda)$,  for a certain highest weight  $L^{HVir}$--module  $L^{HVir} [x,y]$.

Note that $\mu = \frac{I}{2} - J$.  Since $\langle \mu, c \rangle = - \langle \mu, d \rangle = -1$ and
$\nu = -\mu + c$, we have $e ^{\nu} \in \Pi_{-1} (0)$.

Direct calculation yields:

\begin{lemma} \label{racun-1}Assume that $L^{HVir}[x,y]$ is the  irreducible, highest weight $L^{HVir}$--module of highest weight $(x,y)$ and with  highest weight vector $v_{x,y}$. Then we have
$$ L^{sug}_{\mathfrak h_4}  (0) \left( v_{x,y} \otimes e^{r \mu + \lambda c} \right)= \left ( y  +  \frac{x}{2} + \frac{ (1-r) (r + 2 \lambda)}{2} {-} \frac{r}{2} \right )   v_{x,y} \otimes e^{r \mu + \lambda c}. $$
In particular,  $L^{HVir}[x,y] \otimes \Pi_r(\lambda)$ is ${\Z}_{\ge 0}$--graded if and only if $r=1$. 
Moreover,
$$ \Omega \equiv  \left(y + \frac{(x-1) ^2 }{2}\right) \mbox{Id} \quad \mbox{on} \  (L^{HVir}[x,y] \otimes \Pi_1(\lambda))_{top}. $$
\end{lemma}
\begin{proof}
Proof follows directly from the expression (\ref{expr-sug-1}) for $L^{sug} _{\mathfrak h_4}\in L^{HVir} \otimes \Pi$.  Then  $ U= (L^{HVir}[x,y] \otimes \Pi_1(\lambda))_{top}$ is a $\mathfrak h_4$--module. The action of central element $\Omega$ on $U$ is 
$$ \Omega \equiv L^{sug}_{\mathfrak h_4}(0)  + \frac{1}{2} I(0) ^2  -\frac{1}{2} I(0 ) \equiv (y  +  \frac{x-1}{2} -\frac{x-1}{2} + \frac{1}{2} (x-1)^2) \mbox{Id}  = \left(y + \frac{(x-1)^2}{2}\right) \mbox{Id}.  $$
 

\end{proof}

Proof of the following lemma is completely analogous to the proof of irreducibility of relaxed modules from \cite{AKR-2021}.
\begin{lemma}   \label{simple}
Set $M =   L^{HVir} [x,y] \otimes \Pi_{1}(\lambda)$.
\item[(1)] $M$ is almost irreducible, i.e. it is generated by  $M_{top}$  and   every nonzero submodule of $M$ has a nonzero
intersection with $M_{top}$.
\item[(2)] $M $ is an irreducible $V^1 (\mathfrak h_4)$--module if and only if the top component $ M_{top}$ is an irreducible module for the Zhu's algebra $A(V^1 (\mathfrak h_4)) = U(\mathfrak h_4)$ (= the universal enveloping algebra of $\mathfrak h_4$).
\end{lemma}

We  will now present the  main theorem of this section, which gives a realization of the irreducible relaxed modules.

\begin{theorem}\label{thm:relax} Assume that $L^{HVir}[x,y]$ is irreducible, highest weight $L^{HVir}$--module with highest weight $(x,y)$ and highest weight vector $v_{x,y}$. Then
$ L^{HVir} [x,y] \otimes \Pi_{1}(\lambda)$ is relaxed highest weight $V^1 (\mathfrak h_4)$--module.

Module  $L^{HVir} [x,y] \otimes \Pi_{1}(\lambda)$ is irreducible  $V^1 (\mathfrak h_4)$--module  if and only if
 $$  x=1, y\ne 0   \quad \mbox{or} \quad    x\ne 1, \lambda \not{\cong}  \frac{y}{x-1} \quad \mbox{mod} (\Z). $$
In particular, we have: \begin{itemize}

\item[(1)]  
$\widehat{ \mathcal E}_{x-1,[\lambda],y+ \frac{(x-1)^2}{2}     } = L^{Vir} [x, y ] \otimes \Pi_{1}(\lambda)$ for   
$ x \in {\Z} \setminus \{1\}$ and $\lambda \not{\cong}  \frac{y}{x-1} \quad \mbox{mod} (\Z)$.

\item[(2)]  $\widehat{ \mathcal R}_{0,[\lambda],y     } = L^{Vir} [1, y ] \otimes \Pi_{1}(\lambda)$, $y \ne 0$.
\item Modules in (1) and (2) give all irreducible relaxed modules whose top component is neither lowest nor highest weight $\mathfrak h_4$--modules. 
\end{itemize}

We also have the following identification of reducible relaxed modules:
\begin{itemize}
 
\item[(3)]   $\widehat{\mathcal E}^- _{x-1, y+ \frac{(x-1)^2}{2}}  \cong L^{Vir} [x, y ] \otimes \Pi_{1}(\lambda)$ for   $ x \in {\Z} \setminus \{1\}$ and  $\lambda \cong  \frac{y}{x-1} \quad \mbox{mod} (\Z)$. 
\item[(4)]  $\widehat{ \mathcal R}_{0,[\lambda],0     } = L^{Vir} [1, 0 ] \otimes \Pi_{1}(\lambda)$.
 \end{itemize}


\end{theorem}

 \begin{proof}[Proof of Theorem \ref{thm:relax} ]

Let $Z_i =  v_{x,y} \otimes e^{ \mu + (\lambda + i) c}$. Then using  Theorem \ref{inverse-realization}  we get:
\bea  L^{sug} _{\mathfrak h_4} (0) Z_i &=& (y- \frac{x-1}{2} )  Z_i,  \nonumber \\
   E(0)   Z_i &=& Z_{i+1}, \nonumber \\
   F(0) Z_i  & = & \left( y  { -}   (\lambda + i) (x -1) \right) Z_{i-1},   \nonumber \\
   I(0)  Z_i &=& (x-1) Z_i,  \nonumber \\
   J(0) Z_i &=& \frac{1}{2} (x+1 + 2 (\lambda + i)  ) Z_i .   \nonumber
    \eea

Assuming that $y  { -}  (x -1)(\lambda +i) \ne 0$ for all $i \in {\Z}$, we have that top component of $L^{HVir} [x,y] \otimes \Pi_{1}(\lambda)$ is irreducible  module for  $U (\mathfrak h_4)$. Now the simplicity follows from Lemma \ref{simple}.

Next, we want to identify  $M=L^{HVir} [x,y] \otimes \Pi_{1}(\lambda)$ with relaxed highest weight modules from  Theorem  \ref{main-bkrs}.  We have:
$$ I(0) \equiv (x-1) \mbox{Id} \ \mbox{on} \ M_{top}, \quad \Omega \equiv  (y+ \frac{(x-1)^2}{2})  \mbox{Id}  \  \mbox{on} \ M_{top}. $$
Thus $M_{top} \cong \mathcal R_{x-1,  y+ \frac{(x-1)^2}{2} }$.  Therefore 
  $\widehat{ \mathcal E}_{x-1,[\lambda],y+ \frac{(x-1)^2}{2}     } = L^{Vir} [x, y ] \otimes \Pi_{1}(\lambda)$ if $\lambda \not{\cong}  \frac{y}{x-1} \quad \mbox{mod} (\Z)$. 
 
 This proves the assertion (1).
Assume next that  $\lambda   { -}  \frac{y}{x-1}  = -j \in {\Z}$.  Then
  \bea F(0) Z_i &=&    \left( y  { -} (x -1)( \lambda +i)   \right) Z_{i-1}  
   = { (x-1) (j-i)} Z_{i-1}. \nonumber \eea
   This implies that $F(0) Z_j =0$.  This shows that the top component of $ (L^{HVir} [x, y] \otimes \Pi_{1} (\lambda) )_{top}$ contains  lowest weight vector for  the action of $\mathfrak h_4$, and therefore  $(L^{HVir} [x, y] \otimes \Pi_{1} (\lambda) )_{top}  \cong \mathcal R ^-_{i, h}$ for $i= x-1$, $h = y + \frac{(x-1)^2}{2}$. Now  Lemma  \ref{simple} implies that  $L^{HVir} [x, y] \otimes \Pi_{1} (\lambda) \cong  \mathcal E^ - _ {i,h}$, where $\mathcal E^-_{i,h} $ is relaxed module from Section  \ref{affine-nw-voa}.  This proves the assertion (3).

   Now let us prove assertions (2) and (4). For $x=1$ we get
\bea   E(0)   Z_i  =  Z_{i+1},  \ 
   F(0) Z_i   =  y Z_{i-1},  \ I(0) Z_i = 0, \ J(0) Z_i  =  (i+1 + \lambda) Z_i, \label{action-x1} \eea
   which implies that   $ ( L^{HVir} [1, y] \otimes \Pi_{1}(\lambda))_{top}$ is   isomorphic to the  $\mathfrak h_4$--module $\mathcal{R}_{0,[\lambda],y}$. Now  Lemma \ref{simple} implies that $ \widehat{\mathcal{R}}_{0,[\lambda],y}= L^{HVir} [1, y] \otimes \Pi_{1}(\lambda)$, thereby proving isomorphisms in  (2) and (4).     In the case $y \ne 0$, since  $\mathcal{R}_{0,[\lambda],y}$ is an irreducible $\mathfrak h_4$--module for $y \ne 0$, hence  $L^{HVir} [1, y] \otimes \Pi_{1}(\lambda)$ is also irreducible.

   Finally,  by comparing  with Theorem  \ref{main-bkrs}, we get that modules  in (1) and (2) give all irreducible relaxed modules whose top component is neither lowest nor highest weight $\mathfrak h_4$--module.
   
  The proof  of the theorem is now complete.

\end{proof}
{\bf  In the  remainder  of  this section we assume that $x, y, \lambda$  satisfy criteria
such that   relaxed modules $L^{Vir} [x, y] \otimes  \Pi_1(\lambda)$ are reducible.}
\vskip 5mm
Consider the module $L^{HVir} [x, y] \otimes \Pi_{1} (\lambda)$ for $x  \in {\Z}\setminus \{1\}$ and $\lambda   { -}  \frac{y}{x-1}  = -j \in {\Z}$. 
As above, we have  that $F(0) Z_j =0$. The top component of $L^{HVir} [x, y] \otimes \Pi_{1} (\lambda)$ contains conjugate highest weight module $L_1:= V^1(\mathfrak h_4) Z_j$ with conjugate highest weight vector $Z_j$  of  weight
   $$ (x-1, \frac{x+1}{2}    { +} \frac{y}{x-1}). $$

   Next, we consider the quotient module $ L_2 = L^{HVir} [x, y] \otimes \Pi_{1} (\lambda) / L_1$. It is a highest weight module with highest weight vector $\bar Z_{j-1}  = Z_{j-1} + L_1$ of highest weight
   $$ (x-1,   \frac{x-1}{2}  { +}  \frac{y}{x-1} ).$$
  The irreducibility of modules  $L_1$ and $L_2$ follows from Lemma \ref{simple} or by using   Theorem \ref{main-bkrs}. We get:

 \begin{lemma}   \label{str-33} Modules $L_1, L_2$ are irreducible modules, and
 $$ L_1 \cong  \widehat{\mathcal L}^- _{x-1, \frac{x+1}{2}  {  +} \frac{y}{x-1}}, L_2 \cong \widehat{ \mathcal L} ^+_ {x-1,   \frac{x-1}{2}  {  + }\frac{y}{x-1}}. $$
Moreover, $L^{HVir} [x, y] \otimes \Pi_{-1} (\lambda)$ is  the following non-split extension  of $V^1 (\mathfrak h_4) $--modules
   $$ 0 \rightarrow  L_1 \rightarrow  L^{HVir} [x, y] \otimes \Pi_{1} (\lambda)\rightarrow L_2 \rightarrow 0. $$
\end{lemma}

   By applying the automorphism $g$  constructed in Section \ref{affine-nw-voa},  we get  the following consequence of Lemma \ref{str-33}:
   \begin{theorem} \label{str-relaxed-red} 
   We have:
   \item[(1)]  $ L^{HVir} \otimes   \Pi_{r } (\lambda) = g^{r-1} ( L^{HVir} [x, y] \otimes \Pi_{1} (\lambda))$;

  \item[(2)]  $L^{HVir} [x, y] \otimes \Pi_{r} (\lambda)$ is of length two and  it is  the following non-split  extension  of $V^1(\mathfrak h_4)$--modules
   $$ 0 \rightarrow g ^{r-1 } (L_1) \rightarrow   L^{HVir} [x, y] \otimes \Pi_{r} (\lambda)\rightarrow g^{r -1} (L_2) \rightarrow 0. $$

 \item[(3)]  $L^{HVir} [x-1,\frac{x-2}{x-1} y] \otimes \Pi_{r-1} (\lambda)$ is  the following non-split extension  of  $V^1(\mathfrak h_4)$--modules
   $$ 0 \rightarrow g^{r-2}( L_3)  \rightarrow  L^{HVir} [x-1,\frac{x-2}{x-1} y] \otimes \Pi_{r-1} (\lambda) \rightarrow g^{r-2}(L_4) \rightarrow 0, $$
   where ${ L_3 =  \widehat{\mathcal L}^-_{x-2,  \frac{x}{2} +\frac{y}{x-1}} }$,
   $L_4 =  \widehat{\mathcal L}^+_{x-2,  \frac{x-2}{2} +\frac{y}{x-1}}$.

\end{theorem}

Now we consider one important  special case:

\begin{corollary} \label{str-voa-relaxed}
The  $ L^{HVir} \otimes \Pi$--module $L^{HVir} [0, - \lambda] \otimes \Pi_0 (\lambda)$  is  the following non-split extension of $V^1(\mathfrak h_4)$--modules:
\bea  0 \rightarrow  \widehat{ \mathcal L}_{0, \lambda}   \rightarrow  L^{HVir} [0, - \lambda] \otimes \Pi_0 (\lambda)  \rightarrow  g^{-1} (\widehat{\mathcal L}^+_{-1,\lambda  -1/2})  \rightarrow 0. \label{top-1-nonsplit} \eea
{It  is a  cyclic  $V^1(\mathfrak h_4)$--module generated by the  cyclic vector   $v_{0, -\lambda} \otimes e ^{(\lambda-1 )c}$.}
 In particular, the vertex algebra $V^1 (\mathfrak h_4)$ appears in the  following non-split extension:
\bea  0 \rightarrow  V^1 (\mathfrak h_4)  \rightarrow   L^{HVir}  \otimes \Pi \rightarrow  g^{-1} (\widehat{\mathcal L}^+_{-1,  -1/2})  \rightarrow 0. \label{top-2-nonsplit} \eea
 \end{corollary}
 
\begin{proof}  Applying Theorem  \ref{str-relaxed-red}   to the case  $x = r =0$ and letting $y = - \lambda$, we get $L_1 = \widehat{\mathcal L}^- _{-1, \lambda + \frac{1}{2} }$, $L_2  =\widehat{ \mathcal L} ^+_ {-1,  \lambda - \frac{1}{2} }$ and  non-split extension 
  $$ 0 \rightarrow g ^{-1 } (L_1) \rightarrow   L^{HVir} [x, y] \otimes \Pi_{r} (\lambda)\rightarrow g^{ -1} (L_2) \rightarrow 0. $$
Since   $g^{-1} (\widehat{\mathcal L}^- _{-1,   \lambda + 1/2})  = \widehat{\mathcal L}_{0, \lambda}$,
 we  get non-split  extension  (\ref{top-1-nonsplit}). Second assertion follows by specializing $\lambda = 0$, and using the  fact $V^1 (\mathfrak h_4)\cong  \widehat{ \mathcal L}_{0,0}$.
 \end{proof}

\begin{corollary}  \label{str-i0}    Module 
  $\widehat{\mathcal{R}}_{0,[\lambda],0 } =   L^{HVir}[1,  0] \otimes \Pi_{1}(\lambda)$ is reducible 
and for  each $\lambda \in {\C}$ and $i \in {\Z}$, we have the following non-split sequence of  $V^1 (\mathfrak h_4) $--modules:
$$ 0 \rightarrow  \langle Z_{i} \rangle  \rightarrow  \langle   Z_{i-1} \rangle   \rightarrow \widehat {\mathcal  L}_{0,\lambda + i} \rightarrow 0,  $$
where  $Z_i =  v_{1, 0} \otimes e^{ \mu + (\lambda + i) c}$, $\langle Z_i \rangle = V^1 (\mathfrak h_4). Z_i$. 
In particular,    $ L^{HVir}[1,  0] \otimes \Pi_{1}(\lambda)$  has infinite length.
\end{corollary}
 \begin{proof}
  First we notice that for each $i \in {\Z}$:
\bea   E(0)   Z_i  =  Z_{i+1},  \ 
   F(0) Z_i   =  0,   \ I(0) Z_i = 0, \ J(0) Z_i  =  (i+1 + \lambda) Z_i,\nonumber \eea
   which implies that  we have a proper sequence of submodules
   $$ \cdots   \supset \langle Z_{i-1} \rangle \supset   \langle Z_{i} \rangle  \supset    \langle Z_{i+1} \rangle \supset \cdots .   $$
    Since  by Lemma \ref{simple}  $L^{HVir}[1,  0] \otimes \Pi_{1}(\lambda)$ is almost irreducible we  conclude that
  $$ \widehat{ \mathcal  L}_{0,\lambda + i} \cong  \langle  Z_{i-1} \rangle   /    \langle Z_{i} \rangle. $$
  The proof follows.
 \end{proof}
\begin{remark}
The self-dual $HVir$--module $L^{HVir} [1,y]$ for $y\ne 0$ cannot be easily constructed  via  an embedding of $L^{HVir}$ into the rank two Heisenberg vertex algebra (cf.  \cite[Remark 2.2]{AR}). Later, in \cite{AR2}, a modifed realization of  $L^{HVir} [1,y]$ was presented using Whittaker $\Pi$--modules and logarithmic deformations.   This is the underlying reason behind the realization of   $\widehat{\mathcal R}_{0,[\lambda],y} $  for $y \ne 0$ being considerably  more complicated.
\end{remark}


\section{Logarithmic $V^1 (\mathfrak h_4)$--modules}

\subsection{Construction of logarithmic modules}

We shall apply the realization of    $V^1 (\mathfrak h_4)$ inside $L^{HVir} \otimes \Pi$ and the screening operator $S$ for this realization in order to present an explicit realization of logarithmic $V^1 (\mathfrak h_4)$ - modules. We shall use methods developed in \cite{AdM-selecta} and  the fusion rules for $L^{HVir}$--modules obtained in \cite{AR}.

\begin{definition} Let $\mathcal C$ be the category of $V^1 (\mathfrak h_4)$ modules having finite-dimensional generalized weight spaces for $I(0), J(0), L_{\mathfrak h_4} ^{Sug}(0)$, such that $I(0)$ and $J(0)$ act semi-simply.
Let $\mathcal C^{fin}$ be the subcategory of $\mathcal C$ containing modules of finite socle length.
\end{definition}

\begin{remark}
A version of these categories appears in \cite{AKR-2023, ACK-23, FT} in the framework of affine vertex algebras and affine $W$--algebras.  We also believe that $\mathcal C^{fin}$ is  a suitable candidate for obtaining a vertex tensor category structure.
\end{remark}

\vskip 5mm




Now we shall apply the strategy used in the construction of logarithmic modules from Subsection \ref{constr-log-uvod} in the following setting:
\begin{itemize}
\item Let $V = L^{HVir} \otimes \Pi$, $M=  L^{HVir}[-1,0] \otimes \Pi_{-1}(0)$ and $\mathcal V = L^{HVir} \otimes \Pi \oplus  L^{HVir}[-1,0] \otimes \Pi_{-1}(0)$.
\item Let $s=  v_{-1,0} \otimes e^{\nu} \in  L^{HVir}[-1,0] \otimes \Pi_{-1}(0)$. Then $S = s_0 = Res_z Y(s,z)$ is a screening operator.
\item   Since  $V^1 (\mathfrak h_4)  = \mbox{Ker}_ {L^{HVir} \otimes \Pi } { S}  $ and  $S \vert  L^{HVir}[-1,0] \otimes \Pi_{-1}(0) \equiv  0$ by construction of vertex algebra $\mathcal V$ 
we get that
$$   V^1 (\mathfrak h_4)  \subset \mbox{Ker}_{\mathcal V}  S =   V^1 (\mathfrak h_4)  \oplus L^{HVir}[-1,0] \otimes \Pi_{-1}(0) .$$

\end{itemize}


Next we shall use the following fusion rules result from \cite{AR}  when  $x \in {\Z}$ and $y \in {\C}$:
\begin{itemize}
\item $L^{HVir} [-1, 0] \times L^{HVir}[x,y] =  L^{HVir}[x-1, \frac{x-2}{x-1} y] $ for $x \in \Z \setminus \{1\}$;
\item $L^{HVir}[-1,0] \times L^{HVir} [1,y ] = 0$  for $y \ne 0$;
\item $L^{HVir}[-1,0] \times L^{HVir} [1,0 ]  = \sum _{y \in {\C}} L^{HVir}[0,y]$, in the sense that for each $y \in {\C}$ we have non-zero intertwining operator of type
$$ { L^{HVir}[0,y] \choose L^{HVir} [-1, 0] \ \ L^{HVir} [1,0] }. $$

\end{itemize}
This implies that in the category of $L^{HVir} \otimes \Pi$--modules we have the fusion rules:
$$  (L^{HVir} [-1, 0]  \otimes \Pi_{-1}(0)  ) \times  (L^{HVir} [x, y] \otimes \Pi_r(\lambda)) =  L^{HVir} [x-1,  \frac{x-2}{x-1} y] \otimes \Pi_{r-1}(\lambda).  $$
So the candidate for $\mathcal V$--module is  $$U[x,y, \lambda, r] :=L^{HVir} [x, y] \otimes \Pi_r(\lambda) \oplus   L^{HVir} [x-1,  \frac{x-2}{x-1} y] \otimes \Pi_{r-1}(\lambda) \quad (r \in {\Z}, \lambda \in {\C}).$$

Applying Lemma  \ref{AdM-2012}, we see that the necessary and sufficient condition for $U[x,y, \lambda, r] $ to be a  $\mathcal V$--module is that all conformal weights are congruent modulo ${\Z}$. Using Lemma \ref{racun-1} we see that all conformal weights are congruent modulo ${\Z}$ if and only if 
$$ -\frac{y}{x-1} -1 + r + \lambda \in {\Z}. $$
 We get:
  \begin{lemma} Assume that $x \in {\Z}$, $x \ne 1$, $r \in {\Z}$. Then
  $U[x,y, \lambda, r]$ is a  $\mathcal V$--module if and only if  $   \lambda   \equiv { \frac{y}{x-1}   }  \mbox{mod} \  {\Z}$.
  Moreover, the action of the screning operator $S$    on  $U[x,y,  \lambda  , r]$   is given by
 \bea     S \vert _{ U[x,y,  \lambda , r]}= \mbox{Res}_z \mathcal Y (s, z),  \label{scr-act-new} \eea
  where 
  $\mathcal Y$ is the intertwining operator of type
  $${  L^{HVir}[x-1, \frac{x-2}{x-1} y]  \otimes \Pi_{r-1} ( \lambda  )  \choose  L^{HVir} [-1, 0]  \otimes \Pi_{-1}(0)   \ \  L^{HVir} [x, y] \otimes \Pi_r( \lambda)}. $$
  \end{lemma}
  Note that (similarly as in  \cite{A-2017})  in all cases where $U[x,y, \lambda, r]$ is  a $\mathcal V$--module we have that   $L^{HVir} [x, y] \otimes \Pi_r(\lambda)$ is  reducible.

  Applying Proposition \ref{log-koncept}   to the  $\mathcal V$--module  $U[x,y, \lambda, r]$, and using the fact that 
  $V^1(\mathfrak h_4) \subset \mbox{Ker}_{\mathcal V} (S)$, we get the following  logarithmic $V^1(\mathfrak h_4)$--module
  $$(\mathcal P_{r}[x,y], \widetilde Y_{ P_{r}[x,y]} )  := (\widetilde U[x,y,   \frac{y}{x-1}, r], Y_{\widetilde U[x,y,   \frac{y}{x-1}, r]}).$$ 
  This module will be simply denoted by  $\mathcal P_{r}[x,y]$. Recall that the action of the Virasoro algebra on   $\mathcal P_{r}[x,y]$ is
  $$ \widetilde L^{sug} _{\mathfrak h_4} (z) = L^{sug} _{\mathfrak h_4} (z) + z^{-1} Y_{  U[x,y,  \frac{y}{x-1}, r]} ( s, z)  = \sum_{n \in {\Z}} \widetilde L^{sug} _{\mathfrak h_4} (n) z ^{-n-2}. $$
  In particular, we have
  $$ \widetilde L^{sug} _{\mathfrak h_4} (0) = L^{sug} _{\mathfrak h_4}(0) + S,  $$
  where $S$ acts as (\ref{scr-act-new}).
   
  The nilpotent part of   $\widetilde L^{sug} _{\mathfrak h_4} (0)$ is  $\widetilde L^{sug} _{\mathfrak h_4} (0)_{n} = S$, which implies that  $\mathcal P_{r}(x,y)$ has  { rank two.}
  In this way we get:
   \begin{theorem} \label{log-non-split} Assume that  $ x, r,   \lambda - \frac{y}{x-1}   \in {\Z}$ and $x \ne 1$. Then $\mathcal P_{r}[x,y]$ is a logarithmic $V^1 (\mathfrak h_4)$--module which is a non-split extension:
  $$ 0   \rightarrow L^{HVir} [x-1,  \frac{x-2}{x-1} y] \otimes \Pi_{r-1}(\lambda)    \rightarrow   \mathcal P_{r}[x,y] \rightarrow
 L^{HVir} [x, y] \otimes \Pi_r(\lambda)
  \rightarrow 0. $$
  $\mathcal P_{r}[x,y]$  is a logarithmic module of  { rank two}
    with respect to  $\widetilde L^{sug} _{\mathfrak  h_4}(0)$.
  \end{theorem}

  \subsection{Structure of logarithmic modules: finite length case}
\label{str-log-2}
We shall analyze the structure of logarithmic modules constructed in Theorem   \ref{log-non-split}. 

In the case $x=2$, one sees that Theorem \ref{log-non-split} produces a logarithmic module $\mathcal P_r[2,y]$ having  submodule  $ L^{HVir} [1,  0]  \otimes  \Pi_{r-1}(0)$ of infinite  length and whose socle  part is trivial. We shall here  assume that $x \ne 2$.

Using the  structure of relaxed modules $ L^{HVir} [x, y] \otimes \Pi_r(\lambda)$ and $ L^{HVir} [x-1,  \frac{x-2}{x-1} y]  \otimes \Pi_{r-1}(\lambda) $ from Theorem  \ref{str-relaxed-red}, we get
the Loewy diagram of $  \mathcal P_r[x,y]$:

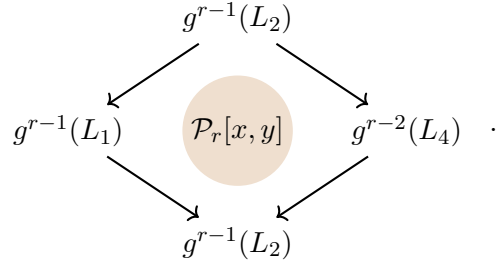
\begin{center}
 \begin{tikzpicture}[baseline=(current bounding box.center)]
\node (tag) at (-6,0) [] {\em{}};
\node (top) at (0,1.5) [] {$ g^{r-1} (L_2) $};
\node (middleI) at (-2.25,0) [] {$g^{r-1}(L_1)$};
\node (middleII) at (2.25,0) [] {$  g^{r-2}(L_4)$};
 \node (bottom) at (0,-1.5) [] {$g^{r-1}( L_2)$};
\draw[->, thick] (top) -- (middleI);
\draw[->, thick] (top) -- (middleII);
\draw[->,thick] (middleI) -- (bottom);
\draw[->, thick] (middleII) -- (bottom);
\node (label) at (0,0) [circle, inner sep=2pt, color=black, fill=brown!25!]
{$ \mathcal P_r[x,y] $};
\end{tikzpicture} $\; .$

\end{center}

Here we use the fact 
$$ g( L_2) = g( \mathcal L^+ _{x-1, \frac{x-1}{2}  + \frac{y}{x-1}} ) = \mathcal L^- _{x-2, \frac{x}{2}  + \frac{y}{x-1}} =L_3.$$
For $r=1$, we get
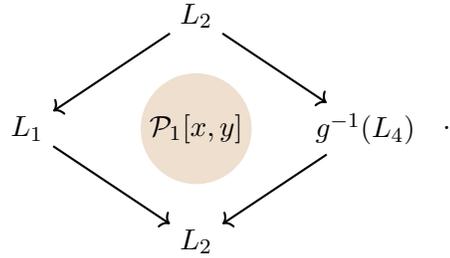
\begin{center}
 \begin{tikzpicture}[baseline=(current bounding box.center)]
\node (tag) at (-6,0) [] {\em{}};
\node (top) at (0,1.5) [] {$  L_2$};
\node (middleI) at (-2.25,0) [] {$L_1$};
\node (middleII) at (2.25,0) [] {$  g^{-1}( L_4)$};
 \node (bottom) at (0,-1.5) [] {$ L_2$};
\draw[->, thick] (top) -- (middleI);
\draw[->, thick] (top) -- (middleII);
\draw[->,thick] (middleI) -- (bottom);
\draw[->, thick] (middleII) -- (bottom);
\node (label) at (0,0) [circle, inner sep=2pt, color=black, fill=brown!25!]
{$ \mathcal P_1 [x,y]$};
\end{tikzpicture} $\; .$
\end{center}
and for $r=2$, we get
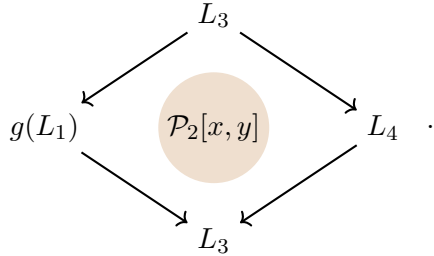
\begin{center}
 \begin{tikzpicture}[baseline=(current bounding box.center)]
\node (tag) at (-6,0) [] {\em{}};
\node (top) at (0,1.5) [] {$  L_3$};
\node (middleI) at (-2.25,0) [] {$g(L_1)$};
\node (middleII) at (2.25,0) [] {$   L_4$};
 \node (bottom) at (0,-1.5) [] {$ L_3$};
\draw[->, thick] (top) -- (middleI);
\draw[->, thick] (top) -- (middleII);
\draw[->,thick] (middleI) -- (bottom);
\draw[->, thick] (middleII) -- (bottom);
\node (label) at (0,0) [circle, inner sep=2pt, color=black, fill=brown!25!]
{$ \mathcal P_2[x,y]$};
\end{tikzpicture} $\; .$
\hspace{3cm} $\text{}$
\end{center}


We have:
 \begin{proposition} Assume that $x \in {\Z}$, $x\notin \{1, 2\}$, $r \in {\Z}$.  The  logarithmic module  $\mathcal P_r[x,y]$ belongs to  $\mathcal C^{fin}$ and it  is of  socle length three. The  socle part is isomorphic to  $$g^{r-1}( \widehat{\mathcal L}^+ _{x-1, \frac{x-1}{2}  + \frac{y}{x-1}}). $$
 In particular,
  the socle part of  $\mathcal P_1[x,y]$ is isomorphic to  $\widehat{ \mathcal L} ^+_ {x-1,   \frac{x-1}{2}  + \frac{y}{x-1}}$.
 
 \end{proposition}
 
 \begin{remark} 
 The structure of the logarithmic module  $\mathcal P_1[x,y]$ shows that it has an analogous  Loewy diagram  as the projective covers in the category of weight $L_k(\mathfrak{sl}_2)$--modules (cf. \cite{A-2017, ACK-23}). Hence is natural to conjecture that  there is a  subcategory $\mathcal C_1 ^{fin}$ of \ $\mathcal C^{fin}$ such that $\mathcal P_1[x,y]$  is the projective cover of   $\widehat{ \mathcal L} ^+_ {x-1,   \frac{x-1}{2}  + \frac{y}{x-1}}$ in  $ \mathcal C_1 ^{fin}$.  
 \end{remark}


\subsection{On logarithmic modules of infinite length  }

Note that in previous subsection we have constructed cyclic  logarithmic modules  in   $\mathcal C^{fin}$ whose  socle part is irreducible  highest weight  module $ \widehat{\mathcal L}^+ _{i, \lambda}$ for each $i \in {\Z} \setminus \{ 0, 1\}$.    But this construction does not cover the case $i=0$ and modules $\widehat{\mathcal L}_{0, \lambda}$, $\lambda  \in {\C}$. Now we shall construct logarithmic modules  $\mathcal P_1 (\lambda)$ whose socle part  is  $\widehat{\mathcal L}_{0, \lambda}$.

 Using the fusion rules and a   conformal weight argument,  one can easily  see that 
$$L^{HVir} [1,0] \otimes \Pi_{r} (\lambda) \oplus L^{HVir} [0,y] \otimes \Pi_{r-1}(\lambda) $$
has the structure of a $\mathcal V$--module if and only if $\lambda \equiv  - y \ \mbox{mod} (\Z)$.  So we can take $y = -\lambda$, and we get the $\mathcal V$--module $$   U [1, 0, \lambda, r]:=L^{HVir} [1,0] \otimes \Pi_{r} (\lambda) \oplus L^{HVir} [0,-\lambda] \otimes \Pi_{r-1}(\lambda). $$
 The screening opetator $S$ acts via 

 \bea     S \vert _{ U[1,0,   \lambda  , r]}= \mbox{Res}_z \mathcal Y (s, z),  \label{scr-act-new-2} \eea
  where 
  $\mathcal Y$ is the intertwining operator of type
  $${  L^{HVir}[0, -\lambda ]  \otimes \Pi_{r-1} (\lambda  )  \choose  L^{HVir} [-1, 0]  \otimes \Pi_{-1}(0)   \ \  L^{HVir} [1, 0] \otimes \Pi_r( \lambda)}. $$

Then  applying again   Proposition \ref{log-koncept}   to the  $\mathcal V$--module  $U [1, 0, \lambda, r]$ we get
the logarithmic $V^1(\mathfrak h_4)$--module
$$(  \widetilde U[1, 0, \lambda, r], \widetilde  Y_{   \widetilde U[1, 0, \lambda, r] }). $$

   Then $\widetilde L^{sug}_{\mathfrak h_4} (0) _{n} =S$ is the nilpotent part of $\widetilde L^{sug} _{\mathfrak h_4}(0)$  and it acts by  (\ref{scr-act-new-2}).
 Consider   the  case $r=1$.  
 
   \begin{proposition} 
 Let  $Z_i ^{(1)}= v_{1,0} \otimes e^{\mu + (\lambda+ i) c}, \quad Z_i  ^{(2)}= v_{0,-\lambda } \otimes e^{(\lambda + i+1) c}. $
Then 
 $\mathcal P_1(\lambda) =   V^1(\mathfrak h_4). Z^{(1)} _{-1} \subset \widetilde U[1, 0, \lambda, 1]$  is logarithmic module of $\widetilde L^{sug} _{\mathfrak h_4}(0)$   {rank  two}
 and has   infinite socle length.  The socle  part of   $\mathcal P_1(\lambda) $ is generated by $Z_{-1} ^{(2)}$ and it  isomorphic to    $ \widehat{\mathcal L}_{0,\lambda}$.
\end{proposition}
\begin{proof}
Note that Corollary   \ref{str-i0} implies that the  socle part of  the   $V^1(\mathfrak h_4)$--module $L^{HVir} [0,-\lambda] \otimes \Pi_{0}(\lambda)$ is isomorphic to $ V^1(\mathfrak h_4). Z_{-1} ^{(2)} \cong \widehat {\mathcal L}_{0, \lambda}$.

Using Lemma \ref{racun-1} we get 
$ L^{sug}_{\mathfrak h_4} (0) Z_i ^{(1)} =0$, $L^{sug} _{\mathfrak h_4} (0) Z_i ^{(2)} = (i+1) Z_i ^{(2)}$.
%
Then  we get $$\widetilde L^{sug} _{\mathfrak h_4} (0) _{n}  Z^{(1)} _{-1}= S Z^{(1)} _{-1} =    \nu  Z^{(2)}_{-1}, \quad (\nu \ne 0), \quad  \widetilde L^{sug} _{\mathfrak h_4} (0) _{n}  Z^{(1)}_{i } = S  Z^{(1)}_{i }  =   0  \quad ( i \ge 0).  $$
Now we have  that $\mathcal P_1(\lambda)$ is logarithmic module of  $\widetilde L^{sug} _{\mathfrak h_4} (0)$  {rank two}
 and the socle part of   $ \mathcal P_1(\lambda) $ is also  isomorphic  to  $\widehat{\mathcal L}_{0, \lambda}$.

Since $ \widetilde E(z) = \widetilde Y( E, z ) = E(z)$, we directly have that  $Z^{(1)} _ i \in \mathcal P_1(\lambda)$, for $i \ge 0$,  which implies that   $\mathcal P_1(\lambda)$ is of infinite length. The proof follows.
\end{proof}
\begin{remark}
Note that  $\mathcal P_1(\lambda)$ does not have  finite socle length,  but  one can construct quotients of $\mathcal P_1(\lambda)$  of arbitrary finite socle lengths. 
Therefore we conclude that  in the category of finite-length modules $\mathcal C^{fin}$, $ \widehat{\mathcal L}_{0,\lambda}$ cannot have a projective cover. In particular, $ V^1 (\mathfrak h_4) = \widehat{\mathcal L}_{0,0}$ does not have projective cover in $\mathcal C^{fin}$.   But we believe that $\mathcal P_1(\lambda)$ is the projective cover of  $ \widehat{\mathcal L}_{0,\lambda}$ in a proper subcategory of $\mathcal C$.
\end{remark}


\section{Conclusion and future work}
\label{zakljucak}

 In this paper we identify the minimal affine $W$-algebra $W^1(\mathfrak h_4, x, F)$  for the Nappi-Witten vertex algebra  $V^1(\mathfrak h_4)$  and prove that it is isomorphic to the Heisenberg--Virasoro vertex algebra $L^{HVir}$ of level zero. We  also construct the inverse Quantum Hamiltonian Reduction (QHR)  which gives a realization 
\bea  V^1(\mathfrak h_4) \hookrightarrow W^1(\mathfrak h_4, x, F) \otimes \Pi.  \label{real-concl} \eea
We prove that all relaxed highest weight $V^1(\mathfrak h_4)$--modules whose top components are neither lowest nor highest weight $\mathfrak h_4$--modules have the form
$ L^{HVir}[x,y] \otimes \Pi_1(\lambda). $

Using the concepts developed for realization of $L_k(\mathfrak{sl}_2)$--modules in \cite{A-2017}, we present here a realization of a large family of logarithmic modules $\mathcal P_r[x,y]$ which are non-split extension of (reducible) relaxed modules. In particular, for $r=1$,   we  construct a logarithmic module $\mathcal P_1[x,y]$ ($x \in {\Z}\setminus \{1, 2\}$) of finite-length,  whose socle part is an irreducible highest weight   $V^1(\mathfrak h_4)$--module.  We believe that  $\mathcal P_1[x,y]$ is projective in a suitable category of weight modules.

We also construct logarithmic modules {$ \mathcal P_1(\lambda)$} of infinite length  whose socle part is anirreducible highest weight module $\widehat {\mathcal L}_{0, \lambda}$. (Note that
$\widehat {\mathcal L}_{0, 0} = V^1(\mathfrak h_4)$).  

Finally, let us present some open problems which we plan to investigate in future work:
 
 \begin{itemize}
 \item[(1)]  In this paper we construct logarithmic $V^1(\mathfrak h_4)$--modules using inverse QHR starting from irreducible highest weight $L^{HVir}$--modules. Therefore we consider  $V^1(\mathfrak h_4)$--modules which can be extended to a weight  $L^{HVir} \otimes \Pi$--modules. Then as in \cite{A-2017, AdM-selecta,  FFHST}  we apply  logarithmic deformations using   screening operator for the realization (\ref{real-concl}).  In  our construction, we deform the non-logarithmic  $L^{HVir}$--modules to get logarithmic $V^1(\mathfrak h_4)$.

A  second approach for the construction of logarithmic modules would be to start with a logarithmic module $M^1 _{log} $ for $L^{HVir}$  and any $\Pi$--module $M^2$. Then the tensor product module $M^1 _{log}  \otimes M^2$ is a logarithmic   $V^1(\mathfrak h_4)$--module.
The paper \cite{AR2} contains a realization of a large  family of logarithmic modules for $L^{HVir}$, which after applying the above  concept gives   a  family of logarithmic modules. We also think that such modules can be analyzed as logarithmic  $V^1(\mathfrak h_4)$--modules. 


\item[(2)] We plan to determine  the fusion rules for $V^1(\mathfrak h_4)$--modules constructed here.

\item[(3)] We believe that the inverse QHR approach can be applied for vertex algebras obtained from  more general non-reductive Lie algebras and Lie  superalgebras.

 \end{itemize}

   \vskip10pt {\footnotesize{}{ }\textbf{\footnotesize{}D.A.}{\footnotesize{}:
Department of Mathematics, Faculty of Science,  University of Zagreb, Bijeni\v{c}ka 30,
10 000 Zagreb, Croatia; }\texttt{\footnotesize{}adamovic@math.hr}{\footnotesize \par}

\vskip10pt {\footnotesize{}{ }\textbf{\footnotesize{}A.B.}{\footnotesize{}:
School of Mathematical Sciences, Raymond and Beverly Sackler Faculty
of Exact Sciences, Tel Aviv University, Israel}\texttt{\footnotesize{} babichenkoandrei@gmail.com}{\footnotesize \par}

\end{document}